\documentclass[12pt]{article}
\usepackage{amsmath, amsthm, amssymb}
\title{Congruences of a square matrix
and its transpose%
\footnotetext{This is the authors' version of a work that was published in Linear Algebra Appl.
389 (2004) 347--353.}}
\author{Roger A. Horn
\\Department of Mathematics, University of
Utah\\ Salt Lake City, Utah
84103, rhorn@math.utah.edu
 \and
Vladimir V. Sergeichuk\thanks{The research was
done while this author was
visiting the University of Utah
supported by NSF grant
DMS-0070503.}
\\ Institute of
Mathematics, Tereshchenkivska
3\\ Kiev, Ukraine,
sergeich@imath.kiev.ua}
\date{}

\begin{document}


\renewcommand{\le}{\leqslant}
\renewcommand{\ge}{\geqslant}
\newcommand{\diagg}%
{\mathop{\rm diag}\nolimits}

\newcommand{\sdotss}%
{\text{\raisebox{-1.8pt}{$\cdot\,$}%
 \raisebox{1.5pt}{$\cdot$}%
\raisebox{4.8pt}{$\,\cdot$}}}

\newcommand{\sdotsss}%
{\text{\raisebox{-2.2pt}{$\cdot\,$}%
 \raisebox{1.7pt}{$\cdot$}%
\raisebox{5.6pt}{$\,\cdot$}}}

\newtheorem{theorem}{Theorem}
\newtheorem{lemma}[theorem]{Lemma}
\newtheorem{remark}[theorem]{Remark}

\maketitle
\begin{abstract}
It is known that any square matrix $A$ over any field is congruent
to its transpose: $A^{\mathrm T}={S}^{\mathrm T}AS$ for some
nonsingular $S$; moreover, $S$ can be chosen such that $S^2=I$,
that is, $S$ can be chosen to be involutory. We show that $A$ and
$A^{\mathrm T}$ are *congruent over any field $\mathbb F$ of
characteristic not two with involution $a\mapsto \bar a$ (the
involution can be the identity): $A^{\mathrm T}=\bar{S}^{\mathrm
T}AS$ for some nonsingular $S$; moreover, $S$ can be chosen such
that $\bar{S}S=I$, that is, $S$ can be chosen to be coninvolutory.
The short and simple proof is based on Sergeichuk's canonical form
for *congruence [\emph{Math. USSR, Izvestiya} 31 (3) (1988)
481--501]. It follows that any matrix $A$ over $\mathbb F$ can be
represented as $A=EB$, in which $E$ is coninvolutory and $B$ is
symmetric.


\end{abstract}

\section{Introduction}
 \label{sec0}

We work over a field $\mathbb F$
of characteristic not two with
involution $a\mapsto \bar{a}$,
that is, a bijection (perhaps
the identity) on $\mathbb F$
such that
\begin{equation*}
 \overline{a+b}= \bar{a}+\bar{b},
 \quad \overline{ab}=\bar{b}
 \bar{a},\quad
 \bar{\bar{a}}=a.
\end{equation*}
For each matrix $A=[a_{ij}]$
over $\mathbb F$, we define
$A^*=\overline{A}^{\mathrm
T}=[\bar{a}_{ji}].$ If $S^*AS=B$
for some nonsingular matrix $S$,
then $A$ and $B$ are said to be
*{\it\!congruent} (or
\emph{congruent} if the
involution $a\mapsto \bar{a}$ is
the identity). Except for
\eqref{777}, all our matrices
are over $\mathbb F$.

In 1980, Gow used Riehm's classification of bilinear forms
\cite{Riehm} to show that any nonsingular square matrix $A$ over
any field is congruent to its transpose: $A^{\mathrm
T}={S}^{\mathrm T}AS$ for some nonsingular $S$; moreover, Gow
showed that $S$ can be chosen such that $S^2=I$, that is, $S$ can
be chosen to be \emph{involutory} \cite{Gow}. Independently at
about the same time, Yip and Ballantine obtained the same theorem
without the hypothesis of nonsingularity \cite{Y&B}. Apparently
unaware of \cite{Gow} and \cite{Y&B},
\raisebox{1pt}{-}\!\!Docovi\'{c} and Ikramov (using Riehm's
classification again (\cite{Riehm} and \cite{R&S-F})) showed in
2002 that $A$ and $A^{\mathrm T}$ are congruent \cite{doc1}.

We are interested in a broader result: Over $\mathbb F$, any
square matrix $A$ is *congruent to $A^{\mathrm T}$; moreover, a
matrix $S$ that gives the *congruence can be chosen such that
$\bar{S}S=I$, that is, $S$ can be chosen to be
\emph{coninvolutory}. Since the involution on $\mathbb F$ can be
the identity, our result includes that of \cite{Y&B} except for
the case of a field of characteristic two.

\section{A canonical form for *congruence}
 \label{sec1}

Our proof that $A$ and
$A^{\mathrm T}$ are *congruent
over $\mathbb F$ is based on the
classification of matrices for
*congruence (up to
classification of Hermitian
matrices) that was obtained in
\cite[Theorem 3]{ser}.

A matrix $M$ is a *{\it\!cosquare} if $ M=A^{-*}A $ for some
nonsingular $A$; $A^{-*}$ denotes $(A^*)^{-1}$. If $M$ is a
*cosquare, every matrix $C$ such that $C^{-*}C=M$ is called a {\it
{\rm *}\!cosquare root} of $M$; we choose any *cosquare root and
denote it by $\sqrt[\displaystyle *]{M}$.

For a polynomial
$f(x)=a_0x^n+a_1x^{n-1}+\dots+a_n\in
{\mathbb F}[x]$ we define
\begin{align*}
\bar{f}(x)&=\bar{a}_0x^n+
\bar{a}_1x^{n-1}+\dots+
\bar{a}_n,\text{\quad and}\\
f^{\vee}(x)
&=\bar{a}_n^{-1}(1+\bar{a}_1x+
\dots+
\bar{a}_{n}x^{n})\text{\quad if
$a_0=1$ and $a_n\ne 0$}.
\end{align*}

Every square matrix is similar
to a direct sum of {\it
Frobenius blocks}
\begin{equation}\label{3}
F_{p^t}=\begin{bmatrix} 0&&
0&-c_n\\1&\ddots&&\vdots
\\&\ddots&0&-c_2\\
0&&1& -c_1 \end{bmatrix},
\end{equation}
in which $p(x)^t=x^n+c_1
x^{n-1}+\dots+ c_n$ is an
integer power of a polynomial
$p(x)$ that is irreducible over
$\mathbb F$.

\begin{lemma}[{\cite[\S 3]{ser}}]
 \label{lemm}
Let $p(x)$ be irreducible over
$\mathbb F$ and let $F_{p^t}$ be
the $n\times n$ Frobenius block
\eqref{3}.

 {\rm(a)}
If ${\cal A}$ is an $n\times n$ matrix over $\mathbb F$ and ${\cal
A} = F_{p^t}^*{\cal A}F_{p^t}$, then ${\cal A}$ is a
\emph{Toeplitz matrix}, that is, ${\cal
A}=[\alpha_{i-j}]_{i,j=1}^n$ for some scalars $\alpha_{1-n},
\ldots, \alpha_{-1}, \alpha_0, \alpha_1, \ldots, \alpha_{n-1}$ in
$\mathbb F$.

 {\rm(b)}
If ${\cal A}$ is a {\rm *}\!cosquare root of $F_{p^t}$, then
${\cal A} = {\cal A}^*F_{p^t}=F_{p^t}^*{\cal A}F_{p^t}$, and so it
is a Toeplitz matrix. Moreover, it has the special form
\begin{equation}\label{10}
{\cal A}=\begin{bmatrix}
  a_0 & a_1&a_2&\cdots &a_{n-1} \\
  \bar{a}_0 & a_0&a_1&\cdots &a_{n-2} \\
  \bar{a}_1 & \bar{a}_0&a_0&\ddots &\vdots \\
  \vdots & \vdots&\ddots &\ddots &a_1\\
  \bar{a}_{n-2}&\bar{a}_{n-3}&\cdots&\bar{a}_0&a_0
\end{bmatrix}.
\end{equation}

 {\rm(c)}
$F_{p^t}$ is a {\rm *}\!cosquare
if and only if
\begin{equation}\label{eeq2}
 \left. \begin{matrix}p(x)\ne x,
  \quad p(x)=p^{\vee}(x),\text{ and}\\
\text{if the involution on $\mathbb F$ is the identity then also
$p(x)\ne x+(-1)^{n+1}$.}
\end{matrix}\right\}
\end{equation}

 {\rm(d)}
Suppose $p(x)$ satisfies the conditions \eqref{eeq2} and let $m$
denote the integer part of $(n-1)/2$. Then one may take
$\sqrt[\displaystyle *]{F_{p^t}}$ to have the form \eqref{10}, in
which $a_0=\dots=a_{m-1}=0$,
\[
a_m= \begin{cases}
    1 & \text{if $n$ is even and
    $p(x)\ne x-\!\!\sqrt[n-1]{1}$},
            \\
p(-1)^t & \text{if $n$ is odd
and $p(x)\ne x+1$,}
            \\
    b-\bar{b} & \text{for any
    $b\in\mathbb
    F$ such that $b\ne\bar{b}$, otherwise,}
  \end{cases}
\]
and $a_{m+1},\dots,a_{n-1}$ are
determined by the identity
$\sqrt[\displaystyle
*]{F_{p^t}}=(\sqrt[\displaystyle
*]{F_{p^t}})^*F_{p^t}$, i.e.,
\begin{equation}\label{11}
\begin{bmatrix}
  \bar{a}_0 & a_0&\cdots&a_{n-2}\\
  \bar{a}_1 & \bar{a}_0&\ddots&\vdots \\
  \vdots & \ddots&\ddots &a_0\\
  \bar{a}_{n-1} &
  \cdots&\bar{a}_1&\bar{a}_0
\end{bmatrix}
 \cdot F_{p^t}=
\begin{bmatrix}
  a_0 & a_1&\cdots&a_{n-1} \\
  \bar{a}_0 & a_0&\ddots&\vdots \\
  \vdots & \ddots&\ddots &a_1\\
  \bar{a}_{n-2} & \cdots&\bar{a}_0&a_0
\end{bmatrix}.
\end{equation}
\end{lemma}

Suppose $F_{p(x)^t}$ is a
*cosquare. Since
$p(x)=p^{\vee}(x)$, the field
\begin{equation}\label{1f}
{\mathbb F}[\kappa]={\mathbb
F}[x]/p(x){\mathbb F}[x],\qquad
\kappa:=x+p(x){\mathbb F}[x],
\end{equation}
possesses the involution
\begin{equation}\label{1g}
f(\kappa)\mapsto
f(\kappa)^{\circ}:=
\overline{f}(\kappa^{-1}).
\end{equation}
It was proved in \cite[Lemma
7]{ser} that if
$f(\kappa)\in{\mathbb
F}[\kappa]$ and $f(\kappa)=
f(\kappa)^{\circ}$, then
$f(\kappa)$ is uniquely
representable as
$f(\kappa)=\varphi(\kappa)$, in
which
\begin{equation}\label{1h}
\varphi(x)=\bar{b}_rx^{-r}+
\bar{b}_{r-1}x^{-r+1} +\dots+
b_0+\dots+b_{r-1}x^{r-1}
+b_rx^r,
\end{equation}
$r$ is the integer part of
$(\deg p(x))/2$,
$b_0,b_1,\dots,b_r\in \mathbb
F$, $b_0=\bar{b}_0$, and if
$\deg p(x)$ is even then
\[
b_r=\begin{cases}
    0 & \text{if the involution $b\mapsto
\bar{b}$ is the identity},
\\
    \bar{b}_r & \text{if $b\mapsto
\bar{b}$ is not the identity and
$p(0)\ne 1$},\\
    -\bar{b}_r &
\text{if $b\mapsto \bar{b}$ is
not the identity and $p(0)=1$.}
  \end{cases}
\]

\begin{theorem}[{\cite[Theorem
3]{ser}}]
 \label{th1}
Let\/ $\mathbb F$ be a field of
characteristic not two with
involution $($the involution can
be the identity$)$. Every square
matrix $A$ over\/ $\mathbb F$ is
{\rm *}\!congruent to a direct
sum of matrices of the three
types:
\begin{itemize}
 \item [{\rm(i)}]
a singular Jordan block
$J_n(0)$;

 \item [{\rm(ii)}]
$\sqrt[\displaystyle *]{F_{p^t}}
\varphi(F_{p^t})$, in which
$F_{p^t}$ is the $n$-by-$n$
Frobenius block \eqref{3},
$p(x)$ satisfies \eqref{eeq2},
and $\varphi(x)$ is a nonzero
function of the form \eqref{1h};

 \item [{\rm(iii)}]
$\begin{bmatrix}0&I_{n}\\F_{p^t}
&0
\end{bmatrix}$,
in which $p(x)\ne x$ and $p(x)$
does not satisfy \eqref{eeq2}.
\end{itemize}
Any matrix of type {\rm(ii)} is
a {\rm *}\!cosquare root of
$F_{p^t}$ and hence is a
Toeplitz matrix. The summands
are determined by $A$ to the
following extent:
\begin{description}
  \item [Type (i)] uniquely.

  \item [Type (ii)]
up to replacement of the whole
group of summands
\[
\sqrt[\displaystyle
*]{F_{p^t}}\varphi_1(F_{p^t})
\oplus\dots\oplus
\sqrt[\displaystyle
*]{F_{p^t}}\varphi_s(F_{p^t})
\]
with the same ${p(x)^t}$ by
\[
\sqrt[\displaystyle
*]{F_{p^t}}\psi_1(F_{p^t})
\oplus\dots\oplus
\sqrt[\displaystyle
*]{F_{p^t}}\psi_s(F_{p^t})
\]
in which each $\psi_i(x)$ is a
nonzero function of the form
\eqref{1h} and the Hermitian
matrices
\begin{equation}\label{777}
\diagg(\varphi_1(\kappa),\dots,\varphi_s(\kappa))
\qquad\text{and}\qquad
\diagg(\psi_1(\kappa),\dots,\psi_s(\kappa))
\end{equation}
over the field\/ ${\mathbb
F}[\kappa]$ defined in
\eqref{1f} with the involution
\eqref{1g} are {\rm
*}\!congruent.

  \item [Type (iii)]
up to replacement of $F_{p^t}$
by $F_{(p^{\vee})^t}$.
\end{description}
\end{theorem}

In a canonical form for similarity, one may choose as the direct
summands (canonical blocks) any matrices that are similar to the
Frobenius blocks \eqref{3}. Over some fields, this freedom of
choice can make it possible to achieve a pleasantly simple and
convenient canonical form for {\rm *}\!congruence. For example, if
$\mathbb F=\mathbb C$ is the field of complex numbers, then the
irreducible polynomials are all of the form $p(x)=x-\lambda$, and
$F_{(x-\lambda)^n}$ is similar to the $n$-by-$n$ Jordan block
$J_n(\lambda)$ with eigenvalue $\lambda$. The conditions
\eqref{eeq2} tell us that when the involution on $\mathbb C$ is
complex conjugation, then $F_{(x-\lambda)^n}$ is a *cosquare if
and only if $|\lambda|=1$; for the identity involution on $\mathbb
C$, $F_{(x-\lambda)^n}$ is a cosquare if and only if
$\lambda=(-1)^{n+1}$.

Define the $n$-by-$n$ matrices
\begin{equation*}\label{1aa}
\Gamma_n = \begin{bmatrix}
\text{\raisebox{-6pt}{\LARGE\rm
0}}&&&&&\text{\raisebox{-4pt}{$\sdotss$}}
\\&&&&1&
\text{\raisebox{-4pt}{$\sdotss$}}\\
&&&-1&-1&\\ &&1&1&\\ &-1&-1&
&&\\ 1&1&&&&\text{\LARGE\rm 0}
\end{bmatrix},\quad
\Delta_n=\begin{bmatrix}
\text{\raisebox{-4pt}{\large\rm
0}}&&&1
\\
&&\text{\raisebox{-4pt}{$\sdotsss$}}&i\\
&1&\text{\raisebox{-4pt}{$\sdotsss$}}&\\
1&i&&\text{\large\rm 0}
\end{bmatrix}.
\end{equation*}
Then $\Gamma_n^{-*}\Gamma_n$ is
similar to $J_n((-1)^{n+1})$ and
$\Delta_n^{-*}\Delta_n$ is
similar to $J_n(1)$. Thus, for
complex matrices we have the
following canonical forms for
congruence and for *congruence
with respect to complex
conjugation:

{\rm(i)} Every square complex
matrix is congruent to a direct
sum, determined uniquely up to
permutation of summands, of
matrices of the form
\begin{equation*}\label{eqqz}
J_n(0),
 \quad
\Gamma_n,
 \quad
\begin{bmatrix}0&I_n\\ J_n(\lambda)
&0
\end{bmatrix},
\end{equation*}
in which $\lambda \ne 0$, $\lambda\ne (-1)^{n+1}$, and $\lambda$
is determined up to replacement by $\lambda^{-1}$.

{\rm(ii)} Every square complex
matrix is {\rm *}\!congruent to
a direct sum, determined
uniquely up to permutation of
summands, of matrices of the
form
\begin{equation*}\label{eqq}
J_n(0),
 \quad
\lambda\Gamma_n,
 \quad
\begin{bmatrix}0&I_n\\ J_n(\mu)
&0
\end{bmatrix},
\end{equation*}
in which $|\lambda|=1$ and\/
$|\mu|>1$. Alternatively, one
may use the symmetric matrix
$\Delta_n$ instead of\/
$\Gamma_n$.

\section{*Congruence
of $A$ and $A^{\mathrm T}$}
 \label{sec2}

The problem of showing that $A$ and $A^{\mathrm T}$ are congruent
has been said to be difficult. In \cite{doc1}, the authors write,
``In spite of its elementary character, the proof of this result
is quite involved.'' In \cite{DSZ} we read that ``The proofs...are
rather complicated.'' However, the difficulty has been in the
methods, not in the results. The canonical forms in Theorem
\ref{th1} permit us to give a short and simple proof of a broader
result.

\begin{theorem}          \label{th2}
Over any field $\mathbb F$ of characteristic not two with
involution $a\mapsto \bar a$ $($the involution can be the
identity$)$, every square matrix $A$ is {\rm *}\!congruent to its
transpose. Moreover, there is a coninvolutory matrix $S$ over
$\mathbb F$ such that $A^{\mathrm T}=S^*AS$.
\end{theorem}

\begin{proof}
Let $A_{\text{\rm can}}=S^*AS$ be a canonical form of $A$ for
*congruence, that is, a direct sum of matrices of the three types
described in Theorem \ref{th1}. If $A_{\text{\rm can}}$ is
*congruent to $A_{\text{\rm can}}^{\mathrm T}$, then $A$ is
*congruent to $A^{\mathrm T}$ since $R^*A_{\text{\rm
can}}R=A_{\text{\rm can}}^{\mathrm T}$ implies
\[
(SR\bar S^{-1})^*A(SR\bar
S^{-1})=A^{\mathrm T}.
\]
Hence it suffices to prove that all matrices of the three types
described in Theorem \ref{th1} are *congruent to their transposes,
and that $R$ can be chosen to be coninvolutory.

Matrices of types (i) and (ii)
are always *congruent to their
transposes since they are
Toeplitz matrices, and for any
Toeplitz matrix $B$ we have
\[
\begin{bmatrix}
\text{\raisebox{-4pt}{\Large
0}}&&1\\&\sdotss&\\[-4pt]1&&
\text{{\Large 0}}
\end{bmatrix}
 B
\begin{bmatrix}
\text{\raisebox{-4pt}{\Large
0}}&&1\\&\sdotss&\\[-4pt]1&&
\text{{\Large 0}}
\end{bmatrix}
 =B^{\mathrm T}.
\]
Notice that the congruence is
achieved via a real involutory
matrix.

For each matrix of type (iii) we
have
\[
\begin{bmatrix}0&S^{-1}\\S^*&0\end{bmatrix}
 \cdot
\begin{bmatrix}0&I\\F&0\end{bmatrix}
 \cdot
\begin{bmatrix}0&S\\S^{-*}&0\end{bmatrix}=
\begin{bmatrix}0&F^{\mathrm T}\\I&0\end{bmatrix}
=\begin{bmatrix}0&I\\F&0\end{bmatrix}^{\mathrm
T},
\]
in which $S$ is any nonsingular
matrix such that
$S^{-1}FS=F^{\mathrm T}$.
However, $S$ always can be
chosen to be symmetric
\cite{tau}, and if we do so then
\[
{\cal
S}=\begin{bmatrix}0&S\\S^{-*}&0
\end{bmatrix} =\begin{bmatrix}0&S\\
\bar{S}^{-1}&0\end{bmatrix}
\]
and $\bar{\cal S}{\cal S}=I$.
\end{proof}

It was proved in \cite[p.\,329]{Gow} that any nonsingular matrix
over a field can be represented as $A=EB$, in which $E$ is
involutory and $B$ is symmetric. As an immediate consequence of
Theorem \ref{th2}, we have the following factorization theorem.

\theoremstyle{corollary}
\newtheorem{corollary}[theorem]{Corollary}
\begin{corollary}
Over any field $\mathbb F$ of characteristic not two with
involution $a\mapsto \bar a$ $($the involution can be the
identity$)$, any square matrix $A$ can be represented as $A=EB$,
in which $E$ is coninvolutory and $B$ is symmetric.
\end{corollary}

\begin{proof}
Theorem \ref{th2} ensures that $A=EA^TE^*$ for some coninvolutory
matrix $E$. Therefore, $\bar{E}A=\bar{E}
EA^TE^*=A^T\bar{E}^T=(\bar{E}A)^T \equiv B$ is symmetric and
$A=EB$.
\end{proof}

\end{document}